\documentclass{amsart}
\usepackage{graphicx}
\usepackage{MnSymbol}
\newtheorem{thm}{Theorem}[section]

\newtheorem{lem}[thm]{Lemma}
\newtheorem{prop}[thm]{Proposition}
\theoremstyle{definition}
\newtheorem{defn}[thm]{Definition}
\theoremstyle{remark}
\newtheorem{rem}[thm]{Remark}
\theoremstyle{remark}

\numberwithin{equation}{section}

\newcommand{\Real}{\mathbb R}

\DeclareMathOperator{\R}{\mathbb{R}}
\usepackage{xcolor}

\begin{document}

\title{Equivalent Moser type inequalities in $\mathbb{R}^2$ and the zero mass case}
\author{D. Cassani \and F. Sani \and C. Tarsi}%
\address{Daniele Cassani, Dip.~di Scienza e Alta Tecnologia, Universit\`a degli Studi dell'Insubria, via Valleggio 11, 22100 Como - ITALY}%
\email{Daniele.Cassani@uninsubria.it}
\address{Federica Sani and Cristina Tarsi, Dip.~di Matematica ``F.~Enriques'', Universit\`a degli Studi di Milano, via C.~Saldini 50, 20133 Milano - ITALY}
\email{Federica.Sani@unimi.it; Cristina.Tarsi@unimi.it}
\thanks{Corresponding author: Daniele.Cassani@uninsubria.it}

\date{\today}%
\begin{abstract}
We first investigate concentration and vanishing phenomena concerning Moser type inequalities in the whole plane which involve complete and reduced Sobolev norms. In particular we show that the critical Ruf inequality is equivalent to an improved version of the subcritical Adachi-Tanaka inequality which we prove to be attained. Then, we consider the limiting space $\mathcal{D}^{1,2}(\mathbb{R}^2)$, completion of smooth compactly supported functions with respect to the Dirichlet norm $\|\nabla\cdot\|_2$, and we prove an optimal Lorentz-Zygmund type inequality with explicit extremals and from which can be derived classical inequalities in $H^1(\mathbb{R}^2)$ such as the Adachi-Tanaka inequality and a version of Ruf's inequality.
\end{abstract}
\maketitle
\section{Introduction}
The classical Moser inequality \cite{M} for $u\in H_0^1(\Omega)$, where $\Omega\subset \mathbb R^2$ is a bounded domain, states
\begin{equation}\label{Mi}
\sup_{\|\nabla u\|_2\leq 1}\int_{\Omega}\left(e^{\beta
u^2}-1\right)~dx\leq M(\beta)|\Omega|,\qquad \beta\leq 4\pi
\end{equation}
where the constant $M(\beta)$ stays bounded provided $\beta\leq 4\pi$ and the supremum becomes infinity when $\beta>4\pi$. Moreover, the functional in \eqref{Mi} is compact as long as $\beta<4\pi$, see \cite{DT,PLL}, and in this context the threshold $\beta=4\pi$ plays the role of the critical Sobolev exponent $2^*:=2N/(N-2)$ which yields in higher dimension $N\geq 3$, see \cite{Ta}, the maximal degree of summability as well as the endpoint of compact embeddings in $L^p(\Omega)$ for functions with membership in $H^1_0(\Omega)$. From the point of view of existence and nonexistence of solutions to PDE, differently from the Sobolev case, the exponent $\beta$ in \eqref{Mi} does not play any role and the critical growth, in terms of threshold between existence and nonexistence of solutions, is represented by the quadratic exponential growth retained by the Orlicz class of functions underlying \eqref{Mi}; see \cite{Adi,FMR} and more recently \cite{dMR,dMR2,druet,IMNPDE}.

\noindent Clearly, as the measure $|\Omega|\to +\infty$ no uniform bound can be retained in \eqref{Mi}. However, by restricting the class of functions in the supremum and considering smooth functions such that $\|\nabla u\|_2\leq 1$ and $\|u\|_2\leq K$, $K>0$, as developed by Cao \cite{Cao} one has
\begin{equation}\label{Ti}
\sup_{\|\nabla u\|_2\leq 1,\:\|u\|_2\leq K}\int_{\mathbb R^2}\left(e^{\beta
u^2}-1\right)~dx\leq C(\beta,K)<\infty,\quad \text{ if } \beta \leq 4\pi(1-m)
\end{equation}
where $m\in (0,1)$.
A further result in this direction was obtained by Adachi-Tanaka in \cite{AT} and which reads as follows: for all $u\in H^1(\mathbb{R}^2)\setminus\{0\}$ one has
\begin{equation}\label{ATeq}
\displaystyle \int_{\mathbb{R}^2}\left(e^{\beta\frac{u^2}{\|\nabla u\|_2^2}}-1\right)\,dx\leq C(\beta)\frac{\|u\|_2^2}{\|\nabla u\|_2^2}\:, \text{ where $C(\beta)<\infty$, as long as  } \beta<4\pi
\end{equation}

\noindent The critical Moser case in which $\beta=4\pi$ remained uncovered until Ruf in \cite{R} established the following inequality
\begin{equation}\label{Ri}
\sup_{\displaystyle\stackrel{ \displaystyle u\in H^1(\mathbb{R}^2)}{\|\nabla u\|_2^2+\:\|u\|_2^2\leq 1}}\int_{\mathbb R^2}\left(e^{\beta
u^2}-1\right)~dx<\infty, \quad \text{ if } \beta\leq 4\pi
\end{equation}
which is sharp in the same sense of \eqref{Mi}, namely the supremum becomes infinity as $\beta>4\pi$.

\noindent Recently, a lot of attention has been given to Trudinger-Moser inequalities whose validity extends to the whole space, see \cite{LR,CT, AY, I, RS, IMN, MS, MSN, LLfree, INW}, and this is motivated by connections with mean field equations and conformal geometry \cite{AC,CY}, two dimensional nonlinear Schr\"{o}dinger and Klein-Gordon equations, see e.g. \cite{ASM,OS, IMM, IMNPDE, DC} and references therein. Further generalizations and closely related problems can be found in \cite{BW,BFM,C,FM2,RT}.

The first part of this work aims to better clarify the notion of criticality related to inequalities \eqref{Ti}, \eqref{ATeq} and \eqref{Ri} in connection with the parameter $\beta$. What qualifies a ``critical problem'' is the possible lack of compactness in connection with:
concentration phenomena, which originates from the invariance under some group action; vanishing phenomena, in presence of unbounded domains, which loosely speaking is the counterpart of concentration, as for example extremal sequences may flatten down on the plane still maintaining positive energy; and eventually mass transportation due to translation invariance, see \cite{PLL}.

\noindent Inequality \eqref{Ri} turns out to be critical with respect to all the above features. Indeed, the analysis carried out in \cite{R,CC} shows that when $\beta=4\pi$ in order to prove that the supremum is attained one has to estimate the non compactness level at which concentration occurs and proving that (normalized) extremals avoid that level. Moreover, the vanishing case has been considered in \cite{I}, where the author proves that inequality \eqref{Ri} is not attained when $\beta$ is sufficiently small.

\noindent On the
contrary, no inequality of the form \eqref{Ti}, \eqref{ATeq} may
hold when $\beta=4\pi$ in the sense of Theorem \ref{MT} below; in particular, no
concentration occurs. Then, we show that inequalities \eqref{Ti},
\eqref{ATeq} are always attained regardless of $\beta$ in the
optimal range and hence also vanishing does not play any role.
More precisely, one has:
\begin{thm}\label{MT}
Let $0\leq\delta<1 $ and $K>0$. Then, for any $\beta\in (0, 4\pi(1-\delta)^{-2})$ there is a
constant $C=C(\beta,\delta, K )>0$ such that
\begin{equation}\label{ineq}
\sup_{\begin{array}{c}
 u\in H^1(\mathbb R^2),\\
\|\nabla u\|_2\leq 1-\delta,\\
\|u\|_2\leq K
\end{array}}\int_{\mathbb R^2}\left(e^{\beta u^2}-1\right)\,dx\leq C
\end{equation}
The inequality is sharp in the sense that for any $0\leq \delta<1$ and for any $K>0$ the supremum in \eqref{ineq} becomes infinity when $\beta=4\pi(1-\delta)^{-2}$.

\vspace{.2cm}

\noindent Moreover, the supremum in \eqref{ineq} is attained.
\end{thm}

\noindent We point out that attainability will be proved by showing that extremals of the energy functional avoid concentration as well as vanishing levels and then applying the compactness result of \cite{IMN}.

Though \eqref{ATeq} and \eqref{Ri} seem so far apart, and the constraint in \eqref{Ri} which involves the complete Sobolev norm apparently necessary to reach the critical case $\beta=4\pi$, surprisingly an improvement of the constant $C(\beta)$ in \eqref{ATeq} yields the equivalence of \eqref{ATeq} and \eqref{Ri} as established in the following
\begin{thm}\label{ReqAT} There exists $C>0$ such that the following inequality holds for all $u \in H^1(\R^2)$ with $\| \nabla u \|_2 \leq 1$
\begin{equation}\label{ATimp}
\int_{\mathbb R^2}\left(e^{\beta
u^2}-1\right)~dx\leq \frac{C}{1-\frac{\beta}{4\pi}}\:\|u\|_2^2,\quad \beta<4\pi
\end{equation}
The constant appearing in the right hand side of \eqref{ATimp}
improves the Adachi-Tanaka constant $C(\beta)$ in \eqref{ATeq}
as $\beta \to 4\pi$.

\noindent Furthermore, consider the Ruf supremum
\begin{equation}\label{Rufsup}
\sup_{\displaystyle\stackrel{\displaystyle u\in H^1(\mathbb{R}^2)}{\|\nabla u\|_2^2+\:\|u\|_2^2\leq 1}}\int_{\mathbb R^2}\left(e^{4\pi
u^2}-1\right)~dx<\infty
\end{equation}
then inequalities  \eqref{Rufsup} and \eqref{ATimp} are equivalent.
\end{thm}

In the second part of this paper we study a limiting situation, the so-called \textit{zero mass case}, see e.g. \cite{Y,BL,DS,FM}. Indeed, complete (suitably weighted by potentials) Sobolev norms cast many nonlinear PDEs, such as nonlinear Schr\"{o}dinger equations and mean field equations \cite{BL}, into coercive variational problems. For instance, let us consider as a prototype the Schr\"{o}dinger operator $-\Delta + m_0V(x)I$, where $V$ is a potential which confines a particle of mass $m_0$ and the energy is given by $E(u)=1/2\|\nabla u\|_2^2+1/2\|\sqrt{m_0 V}\,u\|_2^2$ and which has the weighted Sobolev space $H^1_V(\R^2)$ as a natural function space domain: clearly, $E(u)\to +\infty$ if and only if $\|u\|\to+\infty$.
As $m_0\to 0$, it is natural to wonder what happens when only the Dirichlet norm $\| \nabla \cdot \|_2$ is considered, thus $m_0=0$, and to look for embeddings of the homogeneous space $\mathcal D^{1,2}(\R^N)$, which is defined as the completion of smooth compactly supported functions with respect to $\|\nabla\cdot\|_2$. It is well known that
$$\mathcal D^{1,2}(\R^N) \lhookrightarrow L^{2^*}(\R^N), \quad \text{where } 2^*:= \frac {2N}{N-2},\quad N\geq 3$$
while $N=2$ turns out to be a limiting case since by scaling arguments, see \cite{LiebL,AdF}, one has
\begin{equation}
\label{D12emb}
\mathcal D^{1,2}(\R^2) \nlhookrightarrow L^p(\R^2), \quad p \in [1,+ \infty]
\end{equation}
As a consequence, this rules out exponential integrability and thus any kind of Trudinger-Moser type inequality. In this context, denoting by $u^*$ the decreasing rearrangement of the function $u$,  we prove the following optimal result
\begin{thm}\label{limiting} Let $T>0$ then for all $u\in \mathcal D^{1,2}(\R^2)$ the following inequality holds
\begin{equation}\label{ZMine}
\sup_{0<t\leq T} \frac{u^*(t)-u^*(T)}{\sqrt{\log\frac{T}{t}}}\leq  \frac{1}{\sqrt{4\pi}}\|\nabla u\|_2
\end{equation}
The constant appearing in the right hand side of \eqref{ZMine} is
the best possible and attained, for any fixed $T>0$, by the Moser-type sequence of functions
$$
u_{R,\delta}(x):= \frac 1 {\sqrt {2 \pi}}
\begin{cases}
0\,, & |x|>R \vspace{0.1cm}
\\
\displaystyle{ \frac{ \log \frac R r } {\sqrt{\log \delta}} }\,, &
\displaystyle{\frac R \delta \leq |x| \leq R} \vspace{0.1cm}
\\
\sqrt{\log \delta}\,, & \displaystyle{0 \leq |x| \leq \frac R \delta
}
\end{cases}, \qquad \delta>1, \pi R^2=T
$$
\end{thm}

\noindent We notice that extremals of \eqref{ZMine} turn out to be strongly connected with the scale invariance property of \eqref{ZMine} under the groups action of dilations.

Motivated by Theorem \ref{limiting}, we move our attention from the uniform bound in \eqref{Ti}, \eqref{ATeq} and \eqref{Ri} which underly a suitable Orlicz class of functions, to a finer target function space setting for the embedding of $H^1(\R^2)$, namely

\begin{thm}\label{main4}
For all $u\in H^1(\mathbb{R}^2)$ one has
\begin{equation}\label{limine}
\displaystyle\sup_{T>0}\:\sup_{t\in (0,T]}\frac{u^*(t)}{\sqrt{\frac{4\pi}{T}+\log\frac{T}{t}}}\leq \frac{1}{\sqrt{4\pi}} \: \sqrt{\|\nabla u\|_2^2+\|u\|_2^2}
\end{equation}
and the constant in the right and side of \eqref{limine} is
optimal.

\noindent Moreover, inequality \eqref{limine} yields the critical Ruf inequality with respect to the standard Sobolev norm, namely
\begin{equation}
\label{RufModNorm}
\sup_{\begin{array}{c}
 u\in H^1(\mathbb R^2)\\
\|\nabla u\|_2 + \|u\|_2 \leq 1
\end{array}}\int_{\mathbb R^2} (e^{\beta u^2} -1) \, dx < + \infty,\quad \text{ if and only if }\quad \beta\leq 4\pi
\end{equation}
\end{thm}
\noindent
\noindent We observe that the quantity involved in the left hand side of \eqref{limine} turns out to be characterizing a suitable Zygmund class of functions in unbounded domains, see Section \ref{SecZygmund}.

\noindent On the one hand Theorem \ref{ReqAT} somehow downplays the critical role of the parameter $\beta\leq 4\pi$, on the other hand throws new light on new aspects of criticality as the importance of the constant appearing in the right hand side of Adachi-Tanaka type inequalities which turns out as well to be the key ingredient in proving Theorem \ref{main4} and hence establishing a connection between inequalities \eqref{limine}, \eqref{ATeq} and \eqref{Ri}.

\section{Adachi-Tanaka type inequalities: proof of Theorem \ref{MT}}

\subsection{A counterexample} \label{OptimalRange}
Let $\beta>0$ and define
\begin{equation}
\label{TMfunc}
J_\beta(u):= \int_{\R^2} \left(e^{\beta u^2} -1\right) \, dx, \qquad u \in H^1(\R^2)
\end{equation}
Notice that, the Trudinger-Moser functional $J_\beta$ has some interesting scaling properties, namely for any $u \in H^1(\R^2)$ one has
\begin{equation}
\label{scal1}
J_\beta(au)= J_{a^2 \beta}(u), \quad \forall\, a \in \R
\end{equation}
and
\begin{equation}
\label{scal2}
J_\beta (u_b) = \frac 1{b^2} J_\beta (u), \quad \forall\, b \in \R
\end{equation}
where $u_b(x):=u(bx)$, $x \in \R^2$. Therefore, in order to prove Theorem \ref{MT}, we may assume without loss of generality $\delta=0$ and $K=1$, and we have to prove the following
\begin{equation} \label{ren}
\sigma_{\beta}:=\sup_{\begin{array}{c}
 u\in H^1(\mathbb R^2),\\
\|\nabla u\|_2\leq 1,\\
\|u\|_2\leq 1
\end{array}}
\int_{\mathbb R^2} \left(e^{\beta u^2}  -1\right)  dx  \:
\begin{cases}
\leq C(\beta), & \text{if } 0< \beta < 4\pi
\\
= + \infty, & \text{if } \beta= 4 \pi
\end{cases}
\end{equation}

\noindent The first part of Theorem \ref{MT}, namely $\sigma_{\beta} \leq C(\beta)$ for any $\beta \in (0, 4\pi)$, follows from inequality \eqref{ATeq}. Indeed, from the scaling property \eqref{scal2} one has
\begin{equation*}
\label{AT=Cao}
\sigma_{\beta} =
\sup_{\begin{array}{c}
 u\in H^1(\mathbb R^2) \setminus \{0\},\\
\|\nabla u\|_2\leq 1\\
\end{array}}
\frac 1{\|u\|_2^2}\int_{\mathbb R^2} \left(e^{\beta u^2}-1\right)dx
\end{equation*}
which makes evident that \eqref{ineq} is nothing but the Adachi-Tanaka inequality \eqref{ATeq}.

\noindent However, our main concern here is to prove optimality of \eqref{ren}, i.e. $\sigma_{4 \pi}= + \infty$, which is somehow delicate and can not be obtained from the argument used in \cite{AT} to prove optimality of \eqref{ATeq}. Indeed, Adachi and Tanaka considered in \cite{AT} the so-called \textit{Moser sequence}
\begin{equation}
\label{MosSeq1}
w_n(x)=\frac{1}{\sqrt{2\pi}}\left\{%
\begin{array}{ll}
\displaystyle \frac 1{(\log n)^{1/2}} \: \log \frac 1{|x|}, & \displaystyle \frac 1 n < |x| \leq 1
\\
\displaystyle (\log n)^{1/2}, & \displaystyle 0 \leq |x| \leq \frac 1 n
\end{array}%
\right.
\end{equation}
It is easy to see that $\| \nabla w_n \|_2=1$, $\| w_n\|_2 \to 0$ as $n \to + \infty$ and, as a direct consequence of \eqref{Mi},
\begin{equation} \label{Mos}
\sup_n \int_{\mathbb R^2} \left(e^{4 \pi w_n^2} -1\right) dx \leq C
\end{equation}
Furthermore, since
$$
\int_{\mathbb R^2} \left(e^{4 \pi w_n^2} -1\right) dx \geq
\frac 12
$$
as $n \to +\infty$, we have that
$$\lim_{n \to + \infty} \frac 1{\| w_n\|_2^2} \int_{\mathbb R^2} \left(e^{4 \pi w_n^2} -1\right) dx =+ \infty$$
which gives the optimality of \eqref{ATeq}. Nevertheless, Moser's sequence is not useful to prove optimality of \eqref{ren}, since it satisfies \eqref{Mos}.

We next construct an explicit sequence of functions which realizes
$\sigma_{4 \pi}= + \infty$ and which carries some extra
properties, see Remark \ref{extraprop} below. Let $B_{R_n}$ be the
ball of radius $R_n$, where
$$
R_n:=\frac{\sqrt{\log n}}{\log \log n}\longrightarrow \infty,\quad
\text{ as } n\to \infty
$$
and consider the sequence of radial functions
\begin{equation*}
u_n(x)=\frac{1}{\sqrt{2\pi}}\left\{%
\begin{array}{ll}
\displaystyle \frac{1}{\sqrt{\log n}}\left[1-\frac{\log \log n}{4
\log n}\right]^{1/2}\log\left(\frac{R_n}{|x|}\right),
      & \displaystyle \frac{R_n}{n}<|x|\leq R_n \\
    \\
\displaystyle    \sqrt{\log n} \left[1-\frac{\log \log n}{4 \log
n}\right]^{1/2}, & \displaystyle 0\leq |x|\leq \frac{R_n}{n}
\end{array}%
\right.
\end{equation*}
If we let
$$\lambda_n:=  \Bigl[ 1 - \frac{\log \log n}{4 \log n}\Bigr]^{1/2}$$
then
$$u_n (x):= \lambda_n w_n \Bigl( \frac x {R_n}\Bigr)$$
where $w_n$ is defined by \eqref{MosSeq1}. It is easy to see that $u_n$ is a radially decreasing, $\mathcal C^1$-piecewise function such that
$$\| \nabla u_n \|_2= \lambda_n  \| \nabla w_n \|_2 = \lambda_n \longrightarrow 1^{-}, \quad
\text{ as } n \to \infty$$
whereas
\begin{eqnarray*}
\|u_n\|_2^2 &=& \lambda_n^2 R_n^2 \| w_n\|_2^2 = \lambda_n^2 R_n^2 \frac 1{\log n} \Bigl( \frac 1 4 - \frac 1 {4 n^2} - \frac{\log n}{2n^2}\Bigr)
\end{eqnarray*}
and thus $\|u_n\|_2\to 0$, as $n\to\infty$, since $R_n^2/ \log n \to 0$, as $ n \to \infty$.

\noindent Now, let us estimate from below $J_{4 \pi}(u_n)$. In view of \eqref{scal1} and \eqref{scal2}, we have
\begin{eqnarray*}
J_{4 \pi}(u_n) & = & R_n^2 J_{4 \pi \lambda_n ^2}(w_n) \geq 2 \pi R_n^2 \int_0^{\frac 1 n} \left( e^{4 \pi \lambda_n^2 w_n^2} -1 \right) r \, dr
\\
&=& \pi R_n^2 \: \frac 1 {n^2} \left( e^{2 \lambda_n^2 \log n} -1 \right) = \pi R_n^2 \Bigl( e^{2 \lambda_n^2 \log n - 2 \log n} - \frac 1{n^2} \Bigr)
\\
&=& \pi R_n^2 \Bigl( e^{- \frac 1 2 \log \log n} - \frac 1{n^2} \Bigr) = \pi R_n^2\left(\frac{1}{\sqrt{\log n}}-
\frac{1}{n^2}\right)\\
&=&\pi \frac{\log n}{(\log \log n)^2}\left(\frac{1}{\sqrt{\log
n}}- \frac{1}{n^2}\right)\sim \pi \frac{\sqrt{\log n}}{(\log \log
n)^2}\longrightarrow \infty
\end{eqnarray*}
as $n\to \infty$, and the first part of Theorem \ref{MT} is proved.

\begin{rem} \label{extraprop}
One may still wonder if it is possible to extend the validity of \eqref{ineq} up to the critical case, namely including the borderline value $\beta=4\pi$, by requiring additional conditions. In fact, one may ask: \textit{what if we relax the constraint $\| \nabla u \|_2^2 + \|u\|_2^2 \leq 1$ in \eqref{Ri} requiring just $\| \nabla u\|_2 <1$ and $\| u\|_2 \leq K$?}

\noindent Even in this case, the sequence constructed above serves as a counterexample to show that
$$
\sup_{\begin{array}{c}
 u\in H^1(\mathbb R^2),\\
\|\nabla u\|_2 < 1,\\
\|u\|_2\leq K
\end{array}}
\int_{\mathbb R^2}\left(e^{4\pi u^2}-1\right)dx= +\infty
$$
\end{rem}

\subsection{A compactness result: the supremum is attained} \label{Attainability}

What we have proved so far shows that the reduced constraint

$$\|\nabla u\|_2 \leq 1- \delta\quad\text{ and }\quad \|u\|_2\leq1$$

\noindent to which we refer as the \textit{reduced case} in the sequel, in place of the \textit{Ruf case}

$$\| u\|_S^2:=\|\nabla u\|_2^2+\|u\|_2^2\leq 1$$

\noindent does not allow to cover the critical exponent $\beta=4\pi(1-\delta)^{-2}$ in which one expects the lack of compactness due to concentration phenomena; in particular for $\delta>0$ we stress that, in the reduced case, $\beta=4\pi$ is subcritical. Let
\begin{equation*}
d_{\beta}:=\sup_{\begin{array}{c}
 u\in H^1(\mathbb R^2),\\
\|u\|_{S}\leq 1
\end{array}}\int_{\mathbb R^2}\left(e^{\beta
u^2}-1\right)~dx\leq C(\beta, \tau)<\infty, \quad \forall \beta \in [0, 4 \pi]
\end{equation*}
In \cite{R} Ruf showed that, in the \textit{critical case} $\beta=4 \pi$, the supremum $d_{4 \pi}$ is attained and the most involved and inspiring part of his proof is the estimate of the energy level of any \textit{normalized concentrating sequence}, namely $u_n\in H^1(\R^2)$ which converges weakly to zero, $\|u_n\|_S=1$ whereas $\|u_n\|_{S,\:|x|\geq \rho}\to 0$, as $n\to\infty$ for any $\rho>0$. More precisely, in \cite{R} was proved that
$$d_{cl}(4 \pi) \leq e \pi < d_{4 \pi}$$
where
$$d_{cl} (\beta):= \sup\left\{ \lim_{n \to + \infty}\int_{\mathbb{R}^2}\left(e^{\beta u_n^2}-1\right)\,dx\,|\,\{u_n\}, \|u_n\|_S=1, \text{ is a concentrating sequence}\right\}$$
As observed in \cite{I}, the existence of a maximizer for $d_{\beta}$ is non-trivial even in the subcritical case $\beta < 4 \pi$. In fact, in addition to concentration phenomena, a maximizing sequence $\{u_n\}$ for $d_\beta$ may exhibit a \textit{vanishing behavior}, i.e.
$$
\lim_{n \to + \infty} \| \nabla u_n\|_2=0 \quad \text{ and } \quad \limsup_{n \to + \infty} \|u_n\|_2 >0
$$
Hence, in order to establish the attainability of $d_\beta$, one has to exclude both the concentration and vanishing behavior of maximizing sequences, namely
$$d_{\beta} > \max \left\{d_{cl}(\beta), \, d_{vl}(\beta)\right\}$$
where
$$d_{vl} (\beta):= \sup\left\{\lim_{n \to + \infty}\int_{\mathbb{R}^2}\left(e^{\beta u_n^2}-1\right)\,dx\,|\,\{u_n\}, \|u_n\|_S=1, \text{ is a vanishing sequence}\right\}$$
Actually $d_\beta$ is not attained for sufficiently small $\beta>0$, due to vanishing phenomena as established in \cite[Theorem 1.2]{I}.

In this section we prove that the supremum in inequality \eqref{ineq} is attained and, in contrast with the \textit{Ruf case}, we show that in the \textit{reduced case} vanishing does not play any role, regardless of $\beta$ in the optimal range.

\noindent
Let $\beta \in (0, 4\pi(1-\delta)^{-2})$ and let $u_n\in H^1(\mathbb{R}^2)$ be a maximizing sequence for inequality \eqref{ineq}, namely such that $\|\nabla u_n\|_2\leq 1-\delta$, $\| u_n\|_2\leq K$ and
\begin{equation*}
\sigma_\beta:=\sup_{\begin{array}{c}
 u\in H^1(\mathbb R^2),\\
\|\nabla u\|_2\leq 1-\delta,\\
\|u\|_2\leq K
\end{array}}\int_{\mathbb R^2}\left(e^{\beta u^2}-1\right)dx=\lim_{n\to\infty}\int_{\mathbb R^2}\left(e^{\beta u_n^2}-1\right)\,dx
\end{equation*}
Observe that by the P\'olya-Szeg\"{o} principle (see \cite{PS}) we may assume that there exists a radially decreasing maximizing sequence and thus we can take $u_n\in  H_{rad}^1(\mathbb{R}^2)$, the radial part of $H^1(\mathbb{R}^2)$, such that $u_n\rightharpoonup u$, as $n\to\infty$.
\begin{defn}\label{def_NVS}
Let $K>0$, $\delta\in[0,1)$ and $u_n\in H^1(\mathbb{R}^2)$ such that $\|\nabla u_n\|_2\leq 1-\delta$, $\|u_n\|_2=K$ and
$u_n\rightharpoonup u$, as $n\to\infty$. We say that $\{u_n\}$ is a normalized vanishing sequence if $u=0$ and
$$
\lim_{n\to\infty}\|\nabla u_n\|_2=0
$$
\end{defn}
\begin{lem}\label{P_lem}
Let $u_n\in H^1(\mathbb{R}^2)$ be a  sequence  such that $\|\nabla u_n\|_2\leq 1-\delta$ and $\| u_n\|_2\leq K$. Let $P(s):=e^{\beta s^2}-1-\beta s^2$, then
$$
\lim_{n\to\infty}\int_{\mathbb{R}^2}P(u_n)\,dx=\int_{\mathbb{R}^2}P(u)\,dx$$
\end{lem}
\begin{proof}
We can not apply the compactness lemma of Strauss (see \cite{BL}, Theorem A.I) since
$$\int_{\mathbb{R}^2}\left(e^{\frac{4\pi}{(1-\delta)^2}u_n^2}-1\right)\,dx$$
may fail to be bounded. However, we have the following
\begin{equation}\label{MH}
\lim_{|s|\to\infty}\frac{s^2P(s)}{e^{\frac{4\pi}{(1-\delta)^2}s^2}} =0\quad\text{and}\quad\lim_{s\to 0}\frac{P(s)}{s^2}=0
\end{equation}
and the claim follows directly from \cite{IMN}, in particular by \eqref{MH} hypothesis of Theorem 1.5 in \cite{IMN} are fulfilled.
\end{proof}
\noindent As a consequence we have
\begin{equation}\label{final}
\sigma_\beta= \int_{\mathbb{R}^2}\left(e^{\beta u^2}-1\right)\,dx+\beta\lim_{n\to\infty}\int_{\mathbb{R}^2}(u_n^2-u^2)\,dx
\end{equation}
In particular if $u=0$, which is the case of normalized maximizing vanishing sequences of Definition \ref{def_NVS}, one has
\begin{equation}\label{sigma_v_below}
\sigma_\beta\leq\beta K^2
\end{equation}
\begin{defn}
The vanishing level can be defined as follows
$$\sigma_{vl}(\beta):=\sup\left\{\lim_{n \to + \infty}\int_{\mathbb{R}^2}\left(e^{\beta u_n^2}-1\right)\,dx\,|\,\{u_n\} \text{ is a radially decreasing NVS }\right\}$$
\end{defn}
\begin{rem}
A normalized vanishing sequence (NVS for brevity) can be constructed for example as follows: let $\phi\in\mathcal{C}_c^\infty(\mathbb{R}^2)$ be such that $\|\nabla\phi\|_2\leq 1$ and $\|\phi\|_2=K$. Then the scaling $u_n(x):=\lambda_n\phi(\lambda_n x)$ yields a NVS provided $\lambda_n\to 0$.
\end{rem}
\begin{lem}\label{lev_lem} The following hold:
\begin{itemize}
\item[i)] $\sigma_{vl}(\beta)=\beta K^2$;
\item[ii)] $\sigma_\beta > \sigma_{vl}$
\end{itemize}
\end{lem}
\begin{proof}
Let $\{u_n\}$ be a NVS. Then we have
\begin{multline*}
\lim_{n\to\infty}\int_{\mathbb{R}^2}\left(e^{\beta u_n^2}-1\right)\,dx=\lim_{n\to\infty}\int_{\mathbb{R}^2}\left(e^{\beta u_n^2}-1-\beta u_n^2\right)\,dx+ \beta\lim_{n\to\infty}\|u_n\|_2^2=\beta K^2
\end{multline*}
which is straightforward from Lemma \ref{P_lem} since by assumption the weak limit $u=0$; this proves the first claim.

\noindent Now let $v\in H^1(\mathbb{R}^2)$ be such that $\|\nabla v\|_2\leq 1-\delta$ and $\|v\|_2=K$, then
\begin{eqnarray*}
\int_{\mathbb{R}^2}\left(e^{\beta v^2}-1\right)\,dx &=& \sum_{j=1}^{+\infty}\frac{\beta^j}{j!}\int_{\mathbb{R}^2}v^{2j}\,dx\\
& > & \beta\|v\|_2^2= \beta K^2
\end{eqnarray*}
and hence $\sigma_\beta>\sigma_{vl}(\beta)$ as claimed.
\end{proof}
\noindent From Lemma \ref{lev_lem} and \eqref{sigma_v_below} we can exclude the case in which the weak limit of maximizing sequences is zero. The proof of Theorem \ref{MT} will be complete if we show that $u\neq 0$ is an extremal function. For this purpose let
$$\tau^2:=\lim_{n\to\infty}\frac{\|u_n\|_2^2}{\|u\|_2^2}\geq 1$$
and let $u_\tau(x):=u\left(\frac{x}{\tau}\right)$ so that
\begin{eqnarray*}
\|\nabla u_\tau\|_2 &=& \|\nabla u\|_2\leq\lim_{n\to\infty}\|\nabla u_n\|_2\leq 1-\delta\\
\|u_\tau\|_2^2 &=& \tau^2\|u\|_2^2= \lim_{n\to\infty}\|u_n\|_2^2\leq K^2
\end{eqnarray*}
Let us evaluate
\begin{eqnarray*}
\sigma_\beta &\geq& \int_{\mathbb{R}^2}\left(e^{\beta {u_\tau}^2}-1\right)\,dx\\
&=& \tau^2 \int_{\mathbb{R}^2}\left(e^{\beta {u}^2}-1\right)\,dx\\
&=&  \int_{\mathbb{R}^2}\left(e^{\beta {u}^2}-1\right)\,dx +(\tau^2-1) \int_{\mathbb{R}^2}\left(e^{\beta {u}^2}-1-\beta u^2\right)\,dx+(\tau^2-1)\beta \int_{\mathbb{R}^2}u^2\,dx\\
&=& \sigma_\beta-\beta \lim_{n\to\infty}  \int_{\mathbb{R}^2}\left(u_n^2-u^2\right)\,dx\\
&&\qquad +(\tau^2-1) \int_{\mathbb{R}^2}\left(e^{\beta {u}^2}-1-\beta u^2\right)\,dx+\beta\lim_{n\to\infty}(\|u_n\|_2^2-\|u\|_2^2)\\
&=& \sigma_\beta+(\tau^2-1)\int_{\mathbb{R}^2}\left(e^{\beta {u}^2}-1-\beta u^2\right)\,dx
\end{eqnarray*}
thus necessarily $\tau=1$ and hence $u_n\to u$ in $L^2(\mathbb{R}^2)$ which by \eqref{final} yields an extremal function for $\sigma_\beta$; this completes the proof of Theorem \ref{MT}.

\begin{rem}
Let us stress here that in the \textit{reduced case}, inequality \eqref{ineq} is always attained and this is a consequence of the fact that extremals live above the vanishing critical level independently of $\beta$, in striking contrast with the \textit{Ruf case} in which the vanishing level may stand above the level of concentration, see Remark 2.2 in \cite{I}.
\end{rem}

\begin{rem}
During the preparation of the present paper, we have been informed that Ishiwata, Nakamura and Wadade have recently proved in \cite{INW} the existence of extremal functions for the Adachi-Tanaka inequality in any dimensions. However, here we give a different proof, in fact our argument is an application of the results in \cite{IMN} by Ibrahim, Masmoudi and Nakanishi concerning the two dimensional case. We point out that our argument for the $2$-dimensional case can be adapted to cover the $N$-dimensional case with $N \geq 3$ applying the results recently obtained in \cite{MSN}.
\end{rem}

\section{An improved Adachi-Tanaka inequality: proof of Theorem \ref{ReqAT}}
\label{SectionImprAT}

We first show that Ruf's inequality \eqref{Ri} implies an improved
version of the Adachi-Tanaka inequality \eqref{ATeq}. The
improvement concerns the constant $C(\beta)$ appearing in the
right hand side of \eqref{ATeq} when $\beta \to 4\pi$,
which is the relevant case in order to  cover the
critical exponent in Moser's type results, see Remark \ref{RMKconstRHS} below. More precisely, a deeper
inspection of the Adachi-Tanaka result in \cite{AT}, shows that the constant appearing on the right hand side of \eqref{ATeq} is given by
\begin{equation}\label{ATconst-eps}
C_{\varepsilon}(\beta):=4\pi e^{\frac{\beta}{4\pi}}\max\left\{\frac{\beta
}{4\pi},\frac{e^{\frac{\beta}{4\pi\varepsilon}}}
{ 1 - \frac \beta {4 \pi} (1 + \varepsilon)}\right\}
\end{equation}
where $\varepsilon$ is a parameter that can be
arbitrarily fixed in the interval $(0,
4\pi/\beta-1)$. Hence, the Adachi-Tanaka's constant
actually depends also on the auxiliary parameter $\varepsilon$;
the best choice for this parameter, which minimizes the constant
\eqref{ATconst-eps}, can be easily computed: for any given
$\beta$, it is attained for $\varepsilon=1-\beta/4\pi$
(which is less then $4\pi/\beta-1$). Even with this choice of $\varepsilon$, the constant $C_\epsilon(\beta)$ in \eqref{ATconst-eps} grows exponentially fast as $\beta \to 4 \pi$; in particular, we have (up to constants) the following asymptotic behavior as $\beta \to 4 \pi$
\begin{equation}\label{ATconstasym}
C_{1- \frac \beta {4\pi}} (\beta) \sim C(\beta):= \frac{e^{1/(1-\frac{\beta}{4\pi})}}
{\left(1-\frac{\beta}{4\pi}\right)^2}
\end{equation} Clearly
$$
\displaystyle C(\beta)>\frac{C}{1- \beta/4 \pi},\quad\text{ as }\quad  \beta\to
4\pi
$$ for any fixed $C>0$ and note that the right hand side of the above inequality is
the constant appearing in \eqref{ATimp}.

Let
$\beta \in (0,4 \pi)$ and let $u \in H^1(\R^2)$ be such that $\|
\nabla u \|_2 \leq 1$. Set
$$v:= \sqrt{\frac \beta{4 \pi}} u,$$
so that, from scaling \eqref{scal1} we have
$$J_\beta (u)=J_{4 \pi }(v)$$
Consider also the rescaled function
$$v_{\mu}(x):=v({\mu}x) \qquad \text{ with } {\mu}:=\left( \frac{\|v\|_2^2}{1- \beta/4 \pi} \right)^{\frac 1 2} $$
Then
$$\|v_{\mu}\|_S=\| \nabla v \|_2^2 + \frac 1{{\mu}^2} \|v\|_2^2 = \frac \beta{4 \pi} \|\nabla u\|_2^2 + 1 - \frac \beta{4 \pi} \leq 1$$
and hence, as a consenquence of \eqref{Ri},
$$J_\beta (u) = J_{4 \pi} (v) = {\mu}^2 J_{4 \pi}(v_\mu) \leq {\mu}^2 d_{4 \pi} = \frac{\|v\|_2^2}{1- \beta/4 \pi} d_{4 \pi}$$
and inequality \eqref{ATimp} follows.

Now assume that \eqref{ATimp} holds and let us derive Ruf's inequality \eqref{Ri}. Let $\beta< 4\pi$ and let $u \in H^1(\R^2) \setminus \{0\}$ be such that $\|\nabla u\|_2^2+\|u\|_2^2\leq 1$. Set $\theta:=\|\nabla u\|_2^2 \in (0,1)$ so that $\|u\|_2^2\leq 1-\theta$. We distinguish between two cases accordingly to $0<\theta\leq 1/2$ or $1/2<\theta<1$.

\medskip

\noindent Let $0 < \theta\leq 1/2$ and set $$\widetilde{u}:=\sqrt{2}u$$ so that $\|\nabla \widetilde{u}\|_2^2=2\|\nabla u\|_2^2=2\theta\leq 1$ and $\|\widetilde{u}\|_2^2=2\|u\|_2^2\leq 2$. We have
\begin{equation*}
\int_{\mathbb{R}^2}\left(e^{4\pi u^2}-1\right)\,dx=\int_{\mathbb{R}^2}\left(e^{2\pi \widetilde{u}^2}-1\right)\,dx\leq C\|\widetilde{u}\|_2^2\leq 2C
\end{equation*}
where we have actually used the subcritical Adachi-Tanaka inequality \eqref{ATeq}.

\noindent Let $1/2 < \theta <1$ and set $$u_\theta:=\frac{u}{\sqrt{\theta}}$$ so that $\|\nabla u_\theta\|_2^2=1$ and $\|u_\theta\|_2^2=\frac{\|u\|_2^2}{\theta}$. We have
\begin{equation*}
\int_{\mathbb{R}^2}\left(e^{4\pi u^2}-1\right)\,dx=\int_{\mathbb{R}^2}\left(e^{4\pi\theta u_\theta^2}-1\right)\,dx\leq \frac{d_{4\pi}}{1-\theta}\|u_\theta\|_2^2\leq \frac{d_{4\pi}}{1-\theta}\frac{\|u\|_2^2}{\theta}\leq 2d_{4\pi}
\end{equation*}
where we have used \eqref{ATimp}, with $\beta=4\pi\theta<4\pi$, and $\| u \|_2^2 \leq 1 - \theta$.

\begin{rem}
Notice that, the proof of the equivalence of \eqref{Ri} and \eqref{ATimp} deeply depends on the scaling properties \eqref{scal1} and \eqref{scal2} of the Trudinger-Moser functional. Moreover, by the argument above, it is not possible to deduce Ruf's inequality \eqref{Ri} from the Adachi-Tanaka inequality in its original form \eqref{ATeq}. In fact, if we use \eqref{ATeq} instead of the improved version \eqref{ATimp} then, in the case $\theta > 1/2$, we get
$$\int_{\R^2} (e^{4 \pi u^2} -1) \, dx \leq C_\varepsilon(4 \pi \theta) \: \frac{1- \theta} \theta$$
where
$$C_\varepsilon(4 \pi \theta):=4 \pi \max\left\{\theta \: e^{\theta},
\frac { e^{ \theta \left(\displaystyle  1 + \frac 1 \varepsilon \right)  } }  { 1 - \theta (1+ \varepsilon)  } \right\},
\quad 0<\varepsilon <\frac{1}{\theta}-1$$
as in \eqref{ATconst-eps}. Now, notice that
$\varepsilon=\varepsilon(\theta) \to 0$, as $\theta \to 1^-$ and
$$\sup_{\theta \in (\frac 1 2, 1)} C_\varepsilon (4 \pi \theta) \: \frac{1- \theta} \theta = + \infty$$
\end{rem}
\begin{rem}
\label{RMKconstRHS}
Surprisingly, even if the critical exponent $\beta = 4 \pi$ can not be reached in the Adachi-Tanaka inequality where it appears as a limiting endpoint, a refinement of the constant $C(\beta)$ allows to deduce quite directly the \emph{critical} Ruf inequality \eqref{Ri}. In this respect, the constant $C(\beta)$ appearing in the right hand side of \eqref{ATimp} plays a crucial role and, up to authors best knowledge, it seems this property has not been noticed before.
\end{rem}
\section{The zero mass case: proof of Theorem \ref{limiting}}

Let us briefly recall for the reader convenience the definition of decreasing rearrangement of a function. Let $u:\Omega\subset\mathbb{R}^n \to \mathbb R$ be a measurable function and let
$$
\mu_u(s)=\big|\{x\in \Omega: |u(x)|>s\}\big|, \quad s\geq 0
$$
be the distribution function of $u$. The {\emph{ monotone
decreasing rearrangement}} $u^{\ast}:[0, +\infty)\rightarrow [0,
+\infty]$ of $u$ is defined as the distribution function of
$\mu_{u}$, namely
$$
u^*(t):=\big|\{s\in [0, \infty):\mu_{u}(s)>t\}\big| = \sup \Big\{s>0:\big|\left\{x\in\mathbb
R^n : |u(x)|>s\right\}\big|>t\Big\},\:t\in (0,|\Omega|]
$$
whereas the {\emph{ spherically symmetric rearrangement}}
$u^{\sharp}$ of $u$ is defined as
$$
u^{\sharp}(x)=u^*(\omega_n |x|^n), \quad x\in \Omega^{\sharp}
$$
here $\Omega^{\sharp}$ is the open ball with center in the origin
which satisfies $|\Omega^{\sharp}|=|\Omega|$ and $\omega_n$ is the volume of the unit ball. Clearly, $u^*$ is a
nonnegative, non-increasing and right-continuous function on
$[0,\infty)$ and the \emph{maximal function} $u^{**}$ of the rearrangement $u^*$, defined by
$$u^{**}(t):= \frac 1 t \int_0^t u^*(s) \, ds$$
satisfies $u^* \leq u^{**}$; for basic properties on rearrangements we refer to \cite{PS,K,SK}.

Now we prove inequality \eqref{ZMine} which is a version of the following result by Alvino \cite{AA} in the case of bounded domains, extended to the whole space $\R^2$:
\begin{thm}[Alvino, 1977]
\label{AlvThm}
Let $\Omega \subset \R^2$ be a bounded domain, then the following inequality holds
\begin{equation}
\label{AlvIneq}
\sup_{0 < t \leq |\Omega|} \frac{u^* (t)}{\sqrt{\log \frac{|\Omega|} t}} \leq \frac 1{\sqrt{4 \pi}} \|\nabla u\|_2
\end{equation}
for any $u \in H_0^1(\Omega)$. Moreover, the constant appearing in \eqref{AlvIneq} is the best possible and it is attained when $\Omega$ is a ball.
\end{thm}

Let $u\in \mathcal{D}^{1,2}(\mathbb{R}^2)$, $R>0$ and $|x|\leq R$ then
\begin{eqnarray*}
u^\sharp(|x|)-u^\sharp(R) &=&\int_{|x|}^R -(u^\sharp)'(r)\,dr\\
&\leq & \frac{1}{\sqrt{2\pi}}\sqrt{2\pi \int_{|x|}^R [(u^\sharp)'(r)]^2 r\,dr}\sqrt{\int_{|x|}^R \frac{dr}{r}}\\
&\leq & \frac{1}{\sqrt{2\pi}}\|\nabla u^\sharp\|_2\sqrt{\log \frac{R}{|x|}}\\
&\leq & \frac{1}{\sqrt{2\pi}}\|\nabla u\|_2\sqrt{\log \frac{R}{|x|}}
\end{eqnarray*}
by the Polya-Szeg\"o inequality. Hence by the change of variables $t=\pi |x|^2$ and $T=\pi R^2$, we get
\begin{equation*} \label{eqp131}
u^*(t)-u^*(T)\leq \frac{1}{\sqrt{4\pi}}\|\nabla u\|_2\sqrt{\log \frac{T}{t}},\quad 0<t\leq T
\end{equation*}
and thus inequality \eqref{ZMine}.

\begin{rem}
\label{A->AT}
Inequality \eqref{ZMine} implies the Adachi-Tanaka inequality \eqref{ATeq}. This can be showed following the same arguments in the proof of \eqref{ATeq}, see \cite{AT}.
\end{rem}

The sharpness and attainability of the inequality \eqref{ZMine} in $\mathcal D^{1,2}(\R^2)$ can be deduced from Alvino's inequality in its original form, see Theorem \ref{AlvThm}. In fact, for any fixed $T>0$, if $R>0$ satisfies $T=\pi R^2$ (i.e. $T=|B_R|$ where $B_R$ is the ball of radius $R$ centered at the origin in $\R^2$) then
$$\sup_{u \in H_0^1(B_R)} \: \sup_{0< t \leq T} \frac{u^*(t)}{\sqrt{\log \frac T t}} \leq \sup_{u \in \mathcal D^{1,2}(\R^2)} \: \sup_{0 < t \leq T} \frac{u^*(t) - u^*(T)}{\sqrt{\log \frac T t}}$$
For the convenience of the reader, we show how one can reach the same conclusion by direct computations. For any fixed $T>0$, making the change of variables $T=\pi R^2$ and $t= \pi r^2$,
$$ \sup_{0 < t \leq T} \frac{u^*(t) - u^*(T)}{\sqrt{\log \frac T t}} = \sup_{0<r \leq R} \frac{u^\sharp(r) - u^\sharp(R)}{\sqrt{\log \frac {R^2}{r^2}}}$$
We consider the Moser sequence
\begin{equation}
\label{MoserSeq}
w_k(s):=
\begin{cases}
0\,, &  s<0
\\
\displaystyle{ \frac s{\sqrt k}} \,,& 0 \leq s \leq k \vspace{0.1cm}
\\
\sqrt k \,,& k < s
\end{cases}
\end{equation}
Fixed $R>0$, we define the sequence $\{u_k\}_k \subset  \mathcal D^{1,2} (\R^2)$ of radial non-increasing functions as follows
$$w_k(s)= \sqrt{4 \pi} \: u_k(R e^{-s/2}), \quad |x|^2=R^2e^{-s},$$
By construction,
$$\| \nabla u_k\|_2^2 = \int_{- \infty}^{+ \infty} [\dot{w} (s)]^2 \, ds =1$$
and
$$\sqrt{ 4 \pi} \: \sup_{0<r \leq R} \frac{u^\sharp(r) - u^\sharp(R)}{\sqrt{\log \frac {R^2}{r^2}}} = \sup_{0<s < + \infty} \frac{w_k(s)} {\sqrt s}$$
Now, it is easy to check that
$$\sup_{0<s < + \infty} \frac{w_k(s)} {\sqrt s} = 1 \quad \forall k \geq 1$$
and the sequence $\{u_k\}_k$ enables us to obtain at the same time the sharpness of inequality \eqref{ZMine} and that the best constant in \eqref{ZMine} is attained, for any fixed $T>0$, by the following functions
$$
u_{R,\delta}(x):= \frac 1 {\sqrt {2 \pi}}
\begin{cases}
0\,, & |x|>R \vspace{0.1cm}
\\
\displaystyle{ \frac{ \log \frac R r } {\sqrt{\log \delta}} } \,,&
\displaystyle{\frac R \delta \leq |x| \leq R} \vspace{0.1cm}
\\
\sqrt{\log \delta}\,, & \displaystyle{0 \leq |x| \leq \frac R \delta
}
\end{cases},\qquad \delta>1
$$
which are compactly supported on the ball $B_R$ of measure $T=|B_R|$.

\section{A Zygmund-type inequality in $H^1(\R^2)$: proof of Theorem \ref{main4}} \label{SecZygmund}
Inequality \eqref{limine} represents the embedding of $H^1(\R^2)$ into a Zygmund function space settled in unbounded domains. Indeed, in the case of bounded domains $\Omega \subset \R^2$, the Zygmund space $Z^{1/2}(\Omega)$ consists of all measurable functions $u: \, \Omega \to \R$ such that for some constant $\lambda=\lambda(u)>0$
$$\int_{\Omega} e^{\lambda u^2} \, dx < + \infty$$
we refer to \cite{BR,BS} for classical results on interpolation spaces. The above integral does not satisfy the properties of a norm, however the quantity
\begin{equation}
\label{ZygQNormDomains}
\| u\|_{Z^{1/2}(\Omega)}:= \sup_{t \in (0, |\Omega|]} \frac{u^*(t)}{ \sqrt{ 1+ \log \frac{|\Omega|} t}}
\end{equation}
defines a quasi-norm on $Z^{1/2}(\Omega)$ which is equivalent to a norm. In \cite{CRT2}, the authors investigate the embedding
$$H^1_0(\Omega) \lhookrightarrow  Z^{1/2}(\Omega),$$
proving that
$$\| u\|_{Z^{1/2}(\Omega)} \leq \frac 1{\sqrt{4 \pi}} \| \nabla u \|_2 \quad \forall u \in H_0^1(\Omega)$$
and the constant on the right hand side is sharp, that is it cannot be replaced by a smaller constant. In the same spirit of \cite{CRT,CRT2}, here we analyze the embedding of $H^1(\R^2)$ into a Zygmund space but which is now defined in the whole plane. To this aim, since the definition \eqref{ZygQNormDomains} of the quasi-norm in $Z^{1/2}(\Omega)$ depends crucially on the measure of the domain $\Omega$, we first need to introduce a quasi-norm which is domain independent and well-suited to treat the case of the whole space $\R^2$. Note that this is not a priori obvious and the authors could not find in the literature an explicit definition of quasi-norm in Zygmund spaces on domains with infinite measure. In order to overcome this difficulty, our strategy is the following. Let $Z^{1/2}(\R^2)$ be the space consisting of all measurable functions $u: \, \R^2 \to \R$ for which there exists a constant $\lambda = \lambda (u)>0$ such that
$$\int_{\R^2} (e^{\lambda u^2} -1) \, dx < + \infty$$
If we set
$$\| u\|_{Z^{1/2}}:= \sup_{T>0} \sup_{t \in (0,T]} \frac{u^*(t)} { \sqrt{  \frac {4 \pi}T + \log \frac T t}  }$$
then in analogy with the case of Zygmund spaces on bounded domains, we have the following characterization

\begin{prop}
\label{Zcharact}
A measurable function $u: \, \R^2 \to \R$ belongs to $Z^{1/2}(\R^2)$ if and only if $u \in L^2(\R^2)$ and $\| u\|_{Z^{1/2}} < + \infty$.
\end{prop}
\begin{proof}
Let us first prove the following implication
$$u \in Z^{1/2}(\R^2) \quad \Longrightarrow \quad u \in L^2(\R^2) \;  \text{ and } \; \| u\|_{Z^{1/2}} < + \infty$$
So let $u \in Z^{1/2}(R^2)$ and let $\lambda=\lambda (u) >0$ be such that
$$K=K(u):= \int_{\R^2} (e^{\lambda u^2} -1) \, dx < + \infty$$
Since $Z^{1/2}(\R^2) \subset L^2(\R^2)$, it is enough to show that $ \| u\|_{Z^{1/2}} < + \infty$. From Jensen's inequality, we deduce for any $t>0$
$$e^{\lambda [u^*(t)]^2} -1 \leq e^{\lambda [u^{**}(t)]^2} -1 \leq \frac 1 t \int_0^t (e^{\lambda [u^*(s)]^2} -1) \, ds \leq \frac K t$$
and hence
$$e^{\lambda [u^*(t)]^2} \leq \frac K t +1$$
In particular, for fixed $T>0$ and any $t \in (0, T]$, we may estimate
$$\lambda [u^*(t)]^2 \leq \log \biggl( \frac {K+t} t\biggr) \leq \log \biggl( \frac {K+T} t\biggr) = \log \biggl( \frac {K+T} T\biggr) + \log \frac T t \leq  \frac K T+ \log \frac T t$$
and
$$\| u\|_{Z^{1/2}} \leq \frac 1{\sqrt \lambda} \: \max \biggl\{ 1, \, \sqrt{\frac K {4 \pi}} \biggr \} < + \infty$$
The reverse implication follows using the same arguments introduced in \cite{R} to prove the Trudinger-Moser inequality \eqref{Ri} on $\R^2$. Let $u \in L^2(\R^2)$ be such that $c=c(u):=\| u\|_{Z^{1/2}}< + \infty$ then, for any $\lambda >0$, we have
$$\int_{\R^2}(e^{\lambda u^2(x)} -1) \, dx = \int_0^{+ \infty} (e^{\lambda [u^*(t)]^2} -1) \, dt = \int_0^T + \int_T^{+ \infty} (e^{\lambda [u^*(t)]^2} -1) \, dt$$
Boundedness of the integral on the half-line $[T, + \infty)$ follows from the $L^2$-integrability of $u$ together with the following well know inequality
$$[u^*(t)]^2 = \frac 1 t \int_0^t [u^*(t)]^2 \, ds \leq \frac 1 t \int_0^t [u^*(s)]^2 \, ds \leq \frac{\|u\|_2^2} t$$
namely
\begin{equation}\label{eqp132}
u^*(t)\leq \frac{\|u\|_2}{\sqrt{t}}
\end{equation}
In fact, for any $T>1$ and any $\lambda >0$
\begin{equation*}
\begin{split}
\int_T^{+ \infty} (e^{\lambda [u^*(t)]^2} -1) \, dt & \leq \lambda \| u\|_2^2 + \sum_{k=2}^{+ \infty}\frac{\lambda^k}{k!}  \int_T^{+ \infty}[u^*(t)]^{2k} \, dt
\\ &
\leq \lambda \| u\|_2^2 + \sum_{k=2}^{+ \infty}\frac{\lambda^k}{k!} \frac 1 {(k-1) T^{k-1}} \|u\|_2^{2k}
\\ &
\leq \lambda \| u\|_2^2 + \frac 1 T \sum_{k=2}^{+ \infty}\frac{\lambda^k}{k!} \|u\|_2^{2k} \leq e^{\lambda \|u\|_2^2} < + \infty
\end{split}
\end{equation*}
On the other hand, to estimate the integral on the interval $[0,T]$ we use the boundedness of the $Z^{1/2}$-quasi-norm of $u$. Indeed, for any $T>0$
$$\int_0^T (e^{\lambda [u^*(t)]^2} -1) \, dt \leq \int_0^T e^{\lambda c^2 \bigl( \frac {4 \pi} T + \log \frac  T t\bigr)} \, dt = e^{\lambda c^2 \frac{4 \pi} T} \frac T {1- \lambda c^2}$$
provided $1- \lambda c^2>0$.
\end{proof}

\begin{rem}
The appearence of an $L^2$-integrability condition in the above characterization of $Z^{1/2}(\R^2)$ is quite natural as in the passage from $H_0^1(\Omega)$, with $\Omega \subset \R^2$ bounded, to the space $H^1(\R^2)$. We also bring reader's attention to the paper \cite{CRT5} which concerns Moser type inequalities in Zygmund spaces without boundary conditions and which represented the first step toward the understanding of Zygmund class of functions set on unbounded domains.
\end{rem}

\subsection{Proof of Theorem \ref{main4}} The proof of inequality \eqref{limine} is based on \eqref{ZMine}. In fact, if $u\in H^1(\mathbb{R}^2)$ then, for $T>0$ arbitrarily fixed, by joining \eqref{ZMine} and \eqref{eqp132} we obtain
\begin{eqnarray*}
u^*(t)&=& u^*(t)-u^*(T)+u^*(T)\\
&\leq &  \frac{1}{\sqrt{4\pi}}\|\nabla u\|_2\sqrt{\log \frac{T}{t}}+\frac{\|u\|_2}{\sqrt{T}}
\end{eqnarray*}
By means of the following elementary inequality $ac + bd \leq \sqrt{a^2 + b^2} \: \sqrt{c^2 + d^2}$ which holds for any $a$, $b$, $c$, $d \geq 0$, we get
\begin{equation*}
u^*(t) \leq \frac{1}{\sqrt{4\pi}} \: \sqrt{\frac{4 \pi }{T}+\log \frac{T}{t}} \: \sqrt{\|\nabla u\|_2^2+\|u\|_2^2},\quad 0<t\leq T
\end{equation*}
from which inequality \eqref{limine} follows.

\begin{rem}
Notice that, in contrast with the case of bounded domains, where a
group invariant action seems not to be available, in the case of the whole plane the action of the dilation
$$
u(x) \Rightarrow u(\sqrt \mu \: x)
$$
produces a scale-invariant family of inequalities, all of them equivalent to \eqref{limine}. In fact, let $u\in H^1(\Real^2)$, $\mu>0$, and set
$$
u_{ \mu}(x):=u(\sqrt\mu \: x)
$$
Then $ \|\nabla
u_{\mu}\|_2=\|\nabla u\|_2,  \|u_{\mu}\|_2=\frac{1}{\sqrt \mu}\,
\|u\|_2$ and the decreasing rearrangement scales as
$$
u^{*}_{\mu}(t)=u^{*}(\mu \: t)
$$
so that inequality \eqref{limine} turns into the following
$$
\displaystyle\sup_{T>0}\:\sup_{t\in
(0,\frac{T}{\mu}]}\frac{u_{\mu}^*(t)}{\sqrt{\frac{4\pi}{T}+\log\frac{T}{\mu \: t}}}\leq
\frac{1}{\sqrt{4\pi}} \sqrt{\|\nabla
u_{\mu}\|_2^2+\mu\|u_{\mu}\|_2^2}
$$
or, equivalently
$$
\displaystyle\sup_{\tau>0}\:\sup_{t\in
(0,\tau]}\frac{u_{\mu}^*(t)}{\sqrt{\frac{4\pi}{\mu \: \tau}+\log\frac{\tau}{t}}}\leq
\frac{1}{\sqrt{4\pi}}\sqrt{\|\nabla
u_{\mu}\|_2^2+\mu\|u_{\mu}\|_2^2}
$$
Now, since $H^1(\Real^2)$ is invariant under scaling, we obtain
that inequality \eqref{limine} is actually equivalent to the
following one-parameter family of inequalities
\begin{equation*}\label{limine-mu}
  \displaystyle\sup_{\tau>0}\:\sup_{t\in
(0,\tau]}\frac{v^*(t)}{\sqrt{\frac{4\pi}{\mu \: \tau}+\log\frac{\tau}{t}}}\leq
\frac{1}{\sqrt{4\pi}}\sqrt{\|\nabla v\|_2^2+\mu\|v\|_2^2}
\end{equation*}
\end{rem}

Next we prove optimality of inequality \eqref{limine}.
Similarly to the case of the Alvino-type inequality \eqref{ZMine},
we need to prove the existence of a sequence $\{u_k\}_k
\subset H^1(\R^2)$ satisfying
\begin{equation}
\label{ZygIneqOpt} \frac{\sqrt{4 \pi}}{ \left(\| \nabla u_k\|_2^2 + \|
u_k\|_2^2 \right)^{\frac 1 2}} \: \sup_{R>0} \: \sup_{0<r \leq R} \: \frac{
u_k^\sharp(r)}{\sqrt{ \frac 4{R^2} + \log \frac{R^2}{r^2} }}
\longrightarrow 1\,, \quad \text{ as } \; k \to + \infty
\end{equation}
For this purpose we consider Moser's functions \eqref{MoserSeq} and we define
the sequence $\{u_k\}_k \subset H^1(\R^2)$ of radial
non-increasing functions by means of the following change of
variable
$$w_k(s)= \sqrt{4 \pi} u_k(e^{-s/2}), \quad |x|^2=e^{-s}$$
By construction, $\| \nabla u_k \|_2 =1$,
$$\| u_k \|_2^2 = \frac 1 4 \int_{- \infty}^{+ \infty} w_k^2(s) e^{-s} \, ds = \frac 1 2 \biggl(\frac 1 k - \frac 1{e^k} - \frac 1{k e^k}\biggr) \longrightarrow 0, \quad \text{ as } \; k \to + \infty$$
and
$$\sqrt{4 \pi} \: \sup_{R>0} \: \sup_{0<r \leq R} \: \frac{ u_k^\sharp(r)}{\sqrt{ \frac 4{R^2} + \log \frac{R^2}{r^2} }} = \sup_{R>0} \: \sup_{- \log R^2 \leq s < + \infty} \: \frac{w_k(s)}{\sqrt{ \frac 4{R^2}  + \log R^2 + s}}$$
For any $R>0$
$$\sup_{- \log R^2 \leq s < + \infty} \: \frac{w_k(s)}{\sqrt{ \frac 4{R^2}  + \log R^2 + s}} = f_k(R)$$
where
$$
f_k(R) :=
\begin{cases}
\displaystyle{  \frac{\sqrt k}{ \sqrt{ \frac 4 {R^2} + \log R^2 +k
} } } & R> e^{- k/2} \vspace{0.2cm}
\\
\displaystyle{ \frac { R \sqrt k}{2} } & 0<R \leq e^{- k/2}
\end{cases}
$$
Since
$$\sup_{e^{-k/2} < R} f_k(R) = f_k(2) = \frac {\sqrt k}{ \sqrt{1+ \log 4 + k} } \longrightarrow 1,  \quad \text{ as } \; k \to + \infty$$
and
$$\sup_{0 < R \leq e^{-k/2}} f_k(R) = f_k(e^{- k/2}) = \frac{\sqrt k}{2 e^{k/2}} \longrightarrow 0,  \quad \text{ as } \; k \to + \infty$$
we deduce \eqref{ZygIneqOpt}.
\subsection{From Zygmund-type inequality \eqref{limine} to the critical Ruf inequality}\label{Zyg->AT}
It is easy to see that \eqref{limine} yields \emph{subcritical} Ruf's inequality, i.e. \eqref{Ri} with $\beta < 4 \pi$, and consequently also the following \emph{subcritical} inequality
\begin{equation}
\label{RufModNormSub}
\sup_{\begin{array}{c}
 u\in H^1(\mathbb R^2),\\
\|\nabla u\|_2 + \|u\|_2 \leq 1
\end{array}}\int_{\mathbb R^2} (e^{\beta u^2} -1) \, dx < + \infty, \quad \text{ if } \beta < 4 \pi
\end{equation}
Moreover, the critical exponent for inequality \eqref{RufModNormSub} is still $\beta=4 \pi$. This can be deduced by exploiting the modified Moser sequence of functions
$$m_n(x)= (1- \|w_n\|_2) w_n(x):= \frac{1- \| w_n \|_2} {\sqrt {2 \pi}}\left\{%
\begin{array}{ll}
\displaystyle \frac 1{(\log n)^{1/2}} \: \log \frac 1{|x|}, & \displaystyle \frac 1 n < |x| \leq 1
\\
\displaystyle (\log n)^{1/2}, & \displaystyle 0 \leq |x| \leq \frac 1 n
\end{array}%
\right.
$$
(see \eqref{MosSeq1} for the definition of $w_n$).

\noindent
Inequality \eqref{limine} implies also the Adachi-Tanaka inequality \eqref{ATeq}, though actually we are going to show that \eqref{limine} yields an improved version of the Adachi-Tanaka inequality, in the spirit of \eqref{ATimp} and in turn \eqref{RufModNormSub} with $\beta=4\pi$.

\begin{prop}
\label{PropATquad}
There exists a constant $C>0$ such that for any $u \in H^1(\R^2)$ with $\| \nabla u \|_2 \leq 1$
\begin{equation}
\label{ATcub}
\int_{\R^2} \left( e^{\beta u^2} -1\right) \, dx \leq \frac C{(1- \frac \beta {4 \pi})^2} \: \|u\|_2^2 ,\quad  \beta < 4 \pi
\end{equation}
\end{prop}

\noindent As we have seen in Section \ref{SectionImprAT} (see also Remark \ref{RMKconstRHS}) the constant appearing on the right hand side of Adachi-Tanaka type inequalities plays an important role. In fact, exploiting the linear growth in $(1- \beta/4 \pi)^{-1}$, as $\beta \to 4 \pi$ of the constant appearing in \eqref{ATimp}, one can directly deduce the critical Ruf inequality \eqref{Ri}. Note that the above inequality \eqref{ATcub} improves the classical Adachi-Tanaka inequality \eqref{ATeq} and the improvement concerns exactly the constant on the right hand side, which has a quadratic growth in $(1- \beta/4 \pi)^{-1}$, as $\beta \to 4 \pi$. Following the same arguments as in Section \ref{SectionImprAT}, it is easy to show that \eqref{ATcub} implies the \emph{critical} Ruf inequality with respect to the standard Sobolev norm, namely \eqref{RufModNorm}. Therefore, the proof of Theorem \ref{main4} will be complete once we will prove Proposition \ref{PropATquad}.

\noindent Fix $\beta \in (0, 4 \pi)$ and let
$$\sigma_{\beta, K}:=
\sup_{\begin{array}{c}
\| \nabla u \|_2 \leq 1,\\
\|u\|_2 = K
\end{array}}\int_{\mathbb R^2} (e^{\beta u^2} -1) \, dx
$$
with $K>0$. First, let us take advantage of the scaling property \eqref{scal2} of the Trudinger-Moser functional and note that
\begin{equation}
\label{ATscal}
\sup_{\begin{array}{c}
 u\in H^1(\mathbb R^2) \setminus \{0\},\\
\|\nabla u\|_2\leq 1\\
\end{array}}
\frac 1{\|u\|_2^2}\int_{\mathbb R^2} \left(e^{\beta u^2}-1\right)dx =  \frac 1{K^2} \sigma_{\beta, K}
\end{equation}
The idea is now to deduce \eqref{ATcub} from \eqref{limine} by choosing properly $K>0$ in \eqref{ATscal} in order to gain the uniform
bound stated in \eqref{ATcub}.

\noindent Let $u\in H^1(\Real^2)$ with
$\|\nabla u\|_2\leq 1$ and $\|u\|_2=K$. By \eqref{limine},
\begin{eqnarray*}
 \beta [u^*(t)]^2 &\leq & \frac{\beta}{4\pi}
\left(\frac{4\pi}{T}+\log\frac{T}{t}\right)
\left(\|\nabla u\|_2^2+\right\|u\|_2^2) \\
&\leq& \left(\frac{4\pi}{T}+\log\frac{T}{t}\right)
\left(1+K^2\right)\frac{\beta}{4\pi},  \qquad \hbox{for any }
T>0, \quad t\in (0, T]
\end{eqnarray*}
Let us now choose $K>0$ such that
\begin{equation}
\label{Kchoice}
(1+K^2) \frac \beta {4 \pi} \leq 1-K^2 \quad \text{ and } \quad K<1
\end{equation}
With this choice of $K>0$, we obtain for any $T>0$
\begin{eqnarray*}
\int_0^{T} \left(e^{\beta [u^*(t)]^2}-1 \right) \, dt&\leq&
  \int_0^{T}e^{\beta[u^*(t)]^2} \, dt \leq  e^{\frac {4 \pi} T \: (1-K^2)} \int_0^{T} \Bigl( \frac T t\Bigr)^{1-K^2} \, dt
\\
&=&  \frac T {K^2} \: e^{\frac {4 \pi} T \: (1-K^2)} \leq \frac T {K^2} e^{\frac {4 \pi} T}
\end{eqnarray*}
and in particular, for $T=1$ we get
\begin{equation}
\label{01est}
\int_0^{1} \left(e^{\beta [u^*(t)]^2}-1 \right) \, dt \leq  \frac 1 {K^2} \: e^{4 \pi}
\end{equation}
To estimate the integral on $(1, +\infty)$, we use  \eqref{eqp132}, which yields
$$[u^*(1)]^2 \leq \|u\|_2^2 = K^2$$
and recalling that $0<K<1$, we have
\begin{equation}\label{integ-tau-infty}
\begin{split}
  \int_{1}^{+\infty}\left(e^{\beta [u^*(t)]^2}-1\right ) \, dt & \leq  \int_{1}^{+\infty} \beta [u^*(t)]^2 e^{\beta [u^*(t)]^2} \, dt
  \\ &\leq \beta e^{\beta K^2} \|u\|_2^2\leq 4 \pi e^{4 \pi}
\end{split}
\end{equation}
Combine \eqref{01est} and \eqref{integ-tau-infty} to obtain
\begin{eqnarray*}
  \int_{\Real^2} \left(e^{\beta u^2}-1 \right) \, dx &=&\int_0^{\infty} \left( e^{\beta
  [u^*(t)]^2}-1 \right) \, dt\\
&\leq& \frac 1 {K^2} \: e^{4 \pi} + 4 \pi e^{4 \pi} \leq \frac {e^{4 \pi}(1+ 4 \pi)} {K^2}
\end{eqnarray*}
where again we used the fact $0<K<1$.

\noindent Summarizing, we proved that \eqref{limine} implies
$$\sigma_{\beta, K} \leq  \frac {e^{4 \pi}(1+ 4 \pi)} {K^2}$$
provided $K>0$ satisfies \eqref{Kchoice} and, according to \eqref{ATscal},
$$
\sup_{\begin{array}{c}
 u\in H^1(\mathbb R^2) \setminus \{0\},\\
\|\nabla u\|_2\leq 1\\
\end{array}}
\frac 1{\|u\|_2^2}\int_{\mathbb R^2} \left(e^{\beta u^2}-1\right)dx =  \frac 1{K^2} \sigma_{\beta, K} \leq \frac {e^{4 \pi}(1+ 4 \pi)}{K^4}
$$
Let us now focus on the choice of $K>0$. It is clear that if the value
$$K:= \sqrt{\frac 1 4 \Bigl( 1 - \frac \beta {4 \pi}\Bigr)}$$
satisfies \eqref{Kchoice} then the proof of \eqref{ATcub} is complete. This is in fact the case, since for any $\beta < 4 \pi$ we have that $0< K < 1$ and
$$(1+K^2) \frac \beta {4 \pi} = (1+ K^2) \Bigl[\:  1 - \Bigl( 1 - \frac \beta {4 \pi}\Bigr) \:  \Bigr] = (1+K^2)(1 - 4 K^2) < 1 - 3K^2$$

\section{Final remarks}
\label{SecFinalRemarks}
\begin{rem} \label{RmkScal1}
Due to the scaling property \eqref{scal2} of the Trudinger-Moser functional \eqref{TMfunc}, the bound on the $L^2$-norm appearing in \eqref{ineq} does not affect the optimal range of the exponent, which depends only on the bound on the Dirichlet norm. This phenomenon can also be seen in inequality \eqref{Ri}. In fact, as observed in \cite{AY}, a careful inspection of \cite{R} shows that the Trudinger-Moser inequality \eqref{Ri} is still valid if we replace the standard Sobolev norm, i.e.
$$\|u\|_S^2:= \|\nabla u\|_2^2 + \: \|u\|_2^2,$$
with the equivalent norm
$$\|u\|_{S, \tau}^2:= \|\nabla u \|_2^2 + \: \tau \|u\|_2^2$$
where $\tau>0$. Therefore, the following inequality holds
\begin{equation*}
d_{\beta, \tau}:=\sup_{\begin{array}{c}
 u\in H^1(\mathbb R^2),\\
\|u\|_{S,\tau}\leq 1
\end{array}}\int_{\mathbb R^2}\left(e^{\beta
u^2}-1\right)~dx\leq C(\beta, \tau)<\infty, \quad \forall \beta \in [0, 4 \pi]
\end{equation*}
for any fixed $\tau>0$. Let
\begin{equation*}
d_\beta:= \sup_{\begin{array}{c}
 u\in H^1(\mathbb R^2),\\
\| u\|_S^2 \leq 1
\end{array}}\int_{\mathbb R^2}\left(e^{\beta u^2}-1\right)dx
\end{equation*}
be the supremum of inequality \eqref{Ri}, then the relation between $d_{\beta, \tau}$ and  $d_\beta$ can also be deduced from the scaling property \eqref{scal2} of the Trudinger-Moser functional \eqref{TMfunc} and reads as follows
$$\tau d_{\beta, \tau}= d_{\beta}, \quad \forall\, \tau>0$$
In fact, if we let $u_\tau(x):=u(\sqrt \tau x)$ then $\|u_\tau\|_{S,\tau}= \|u\|_S$ and
$$J_\beta(u) = \tau J_{\beta}(u_\tau)$$
\end{rem}

\begin{rem}
The nature of the constraint appearing in \eqref{Ti}, \eqref{ATeq} and \eqref{ineq}, which essentially involves only the Dirichlet norm $\| \nabla \cdot \|_2$, naturally suggests to investigate reasonable embeddings of the limiting homogeneous space $\mathcal D^{1,2}(\R^2)$. As pointed out in the Introduction, any kind of Trudinger-Moser type inequality \emph{cannot} hold to be true in the whole space $\mathcal D^{1,2}(\R^2)$, since it would imply that some exponential growth is allowed in contradiction with \eqref{D12emb}. However, if we consider subspaces of $\mathcal D^{1,2}(\R^2)$ of the form $\mathcal D^{1,2}(\R^2) \cap L^p(\R^2)$ with $p \geq 1$, then for any $\alpha \in (0, 4\pi)$
\begin{equation}
\label{D12AT}
\int_{\R^2} \phi_p(\alpha u^2) \leq C_{\alpha} \|u\|_p^p \quad \forall u \in \mathcal D^{1,2} (\R^2) \cap L^p(\R^2) \text{ with } \| \nabla u \|_2 \leq 1
\end{equation}
as shown in \cite[Corollary 2.1]{KW} in the more general framework of fractional homogeneous Sobolev spaces (see also \cite{BCLS}). Here $\phi_p$ represents the modified exponential function
$$\phi_p(t):= \sum_{\begin{array}{c}
j \in \mathbb N,\\
j \geq p/2
\end{array}}
\frac{t^j}{j!}
$$
In view of the Adachi-Tanaka type inequality \eqref{D12AT} in $\mathcal D^{1,2}(\R^2) \cap L^p(\R^2)$, it is natural to look for a version of the Trudinger-Moser inequality \eqref{Ri} in the same subspaces of $\mathcal D^{1,2}(\R^2)$. Up to our knowledge, no inequality of the form
\begin{equation}
\label{D12+Lp}
\sup_{\begin{array}{c}
 u\in \mathcal D^{1,2}(\mathbb R^2) \cap L^p(\R^2),\\
\|u\|_{\mathcal D^{1,2} \cap L^p} \leq 1
\end{array}}
\int_{\mathbb R^2} \phi_p(4 \pi u^2) \,dx \leq C_p
\end{equation}
with $p \geq 1$ and
$$\|u\|_{\mathcal D^{1,2} \cap L^p}:= \|\nabla u \|_2 + \|u\|_p \quad ( \text{or } \; \|u\|_{\mathcal D^{1,2} \cap L^p}^2:= \| \nabla u\|_2^2 + \|u\|_p^2 \:)$$
appears in the literature. However, we point out that exploiting the same arguments as developed by Lam and Lu in \cite{LL} (see also \cite{LLfree} for the more general case of Sobolev spaces involving higher order derivatives), it is easy to show that \eqref{D12+Lp} holds true in $D^{1,2}(\R^2) \cap L^p(\R^2)$ with $p \geq 1$.
\end{rem}

\begin{rem} It remains open the problem wether inequality \eqref{limine} turns out to be attained as in the Ruf case or if the loss of invariance properties under a group action may prevent attainability as observed in a closely related situation in \cite{CRT3}.
\end{rem}

\bibliographystyle{amsplain}

\end{document}